\documentclass[11pt]{amsart}

\usepackage{amsmath,amssymb,amsthm,amsfonts,mathrsfs}
\usepackage{fullpage}
\usepackage[shortlabels]{enumitem}
\usepackage[colorlinks=true,
linkcolor=blue,
anchorcolor=blue,
citecolor=red
]{hyperref}
\allowdisplaybreaks 
\newtheorem{theorem}{Theorem}[section]

\newtheorem{proposition}[theorem]{Proposition}

\theoremstyle{definition}

\theoremstyle{remark}
\newtheorem{remark}[theorem]{Remark}

\newcommand{\C}{\mathbb{C}}
\newcommand{\A}{\mathbb{A}}

\newcommand{\GL}{\mathrm{GL}}

\newcommand{\SL}{\mathrm{SL}}
\newcommand{\SO}{\mathrm{SO}}
\newcommand{\Sp}{\mathrm{Sp}}
\newcommand{\GSp}{\mathrm{GSp}}
\newcommand{\on}{\operatorname}

\renewcommand{\Re}{\on{Re}}

\author[P. Yan]{Pan Yan}
\address{Department of Mathematics, The University of Arizona, Tucson, AZ 85721, USA}
\email{panyan@math.arizona.edu}

\makeatletter
\@namedef{subjclassname@2020}{\textup{}2020 Mathematics Subject Classification}
\makeatother
\date{\today}
\title{A note on a Hecke type integral for $\Sp(2n)\times \GL(1)$}

\begin{document}

\begin{abstract} 
Let $\pi$ be an irreducible cuspidal automorphic generic representation of $\Sp_{2n}(\A)$ and let $\chi:F^\times\backslash \A^\times\to \mathbb{C}^\times$ be a unitary idele class character. In this 
note,
we present a Rankin-Selberg integral of Hecke type for the twisted standard partial $L$-function $L^S(s,\pi\times\chi)$ of degree $2n+1$.
\end{abstract}

\maketitle

\section{Introduction}
\label{section-introduction}
Zeta functions or $L$-functions are one of the fundamental objects of interest in number theory. A common way to study $L$-functions is through an integral representation, which has its origin in Riemann's paper on the Riemann $\zeta$-function. For the $L$-function $L(s, f)=\sum_{n=1}^\infty a_n {n^{-s}}$ associated to a cusp form $f(z)=\sum_{n=1}^\infty a_n e^{2\pi i n z}$ on the upper half plane, Hecke \cite{MR1513069} used the Mellin transform of $f$ to establish an integral representation
\begin{equation}
\label{eq-Hecke-classical-integral}
    \Lambda(s, f)=\int_0^\infty f(iy)y^s d^\times y=\int_0^\infty f\left(\begin{pmatrix} y & \\ & 1\end{pmatrix}\cdot i\right) y^s d^\times y,
\end{equation}
where $\Lambda(s, f)=(2\pi)^{-s}\Gamma(s)L(s, f)$ is the completed $L$-function, and $\begin{pmatrix} y & \\ & 1\end{pmatrix}\cdot i$ denotes the action by linear fractional transformation. This allowed Hecke to derive analytic properties of $\Lambda(s, f)$.

$L$-functions of automorphic representations were first developed by Jacquet and Langlands \cite{MR0401654} for $\GL_2$. Let $F$ be a number field with the ring of adeles $\mathbb{A}$. Let $(\pi, V_\pi)$ be an irreducible automorphic cuspidal representation of $\GL_2(\A)$ and let $\chi:F^\times\backslash \A^\times\to \C^\times$ be a unitary idele class character. Let $\psi:F\backslash \A\to \C^\times$ be a non-trivial additive character. Let $\varphi\in V_\pi$ be a non-zero cusp form. Following Hecke's integral \eqref{eq-Hecke-classical-integral} where we embed $\GL_1$ into the Levi subgroup of the mirobolic subgroup of $\GL_2$,
Jacquet and Langlands considered the integral 
\begin{equation}
\label{eq-GL(2)xGL(1)-global-integral}
I_2(s, \varphi, \chi)=\int\limits_{F^\times\backslash \A^\times} \varphi\begin{pmatrix}
t & \\
& 1
\end{pmatrix} \chi(t)|t|^{s-1}d^\times t.
\end{equation}
Replacing the cusp form $\varphi$ by its Fourier expansion
\begin{equation*}
    \varphi(g)=\sum_{\gamma\in F^\times } W_{\varphi, \psi} \left( \begin{pmatrix}
    \gamma & \\ & 1
    \end{pmatrix} g\right)
\end{equation*}
where $W_{\varphi, \psi}$ is the Whittaker function of $\varphi$ given by
\begin{equation*}
    W_{\varphi, \psi}(g)=\int\limits_{F\backslash \A}\varphi\left(\begin{pmatrix} 1 & x\\ & 1\end{pmatrix}g\right)\psi^{-1}(x)dx,
\end{equation*}
we see that the integral $I_2(s, \varphi, \chi)$ unfolds to (for $\Re(s)\gg 0$)
\begin{equation}
\label{eq-GL(2)xGL(1)-global-integral-unfolding}
    I_2(s, \varphi, \chi)=\int\limits_{\A^\times} W_{\varphi, \psi}\begin{pmatrix}
t & \\
& 1
\end{pmatrix} \chi(t)|t|^{s-1}d^\times t.
\end{equation}
By the uniqueness of the Whittaker model \cite{MR348047}, the Whittaker function $W_{\varphi, \psi}$ is factorizable and consequently the integral $I_2(s, \varphi, \chi)$ is Eulerian. This allows us to factor the integral \eqref{eq-GL(2)xGL(1)-global-integral-unfolding} and study the local integral corresponding to \eqref{eq-GL(2)xGL(1)-global-integral-unfolding}, and through the unramified computation we can calculate the local unramified integral explicitly to get the local $L$-function. This process gives us an integral representation for $L^S(s, \pi\times \chi)$ where $S$ is a finite set of places outside of which the data are unramified and normalized.

\begin{remark}
For $\SL_2\times\GL_1$, one can modify the integral \eqref{eq-GL(2)xGL(1)-global-integral}, to obtain an integral representation for $L(s, \pi\times\chi)$ where $\pi$ is an irreducible automorphic cuspidal representation of $\SL_2(\A)$ as follows. For $\varphi\in V_\pi$ a non-zero cusp form on $\SL_2(\A)$, we consider the integral
\begin{equation}
\label{eq-SL(2)xGL(1)-global-integral}
J(s, \varphi, \chi)=\int\limits_{F^\times\backslash \A^\times} \varphi\begin{pmatrix}
t & \\
& t^{-1}
\end{pmatrix} \chi(t)|t|^{s-1}d^\times t.
\end{equation}
and it can be shown that (for $\Re(s)\gg 0$)
\begin{equation*}
J(s, \varphi, \chi)= \int\limits_{\A^\times} W_{\varphi, \psi}\begin{pmatrix}
t & \\
& t^{-1}
\end{pmatrix} \chi(t)|t|^{s-1}d^\times t,
\end{equation*}
where $W_{\varphi, \psi}$ is the Whittaker function of $\varphi$. Here, in \eqref{eq-SL(2)xGL(1)-global-integral}, we embed $\GL_1$ into the Levi subgroup of the standard Borel subgroup of $\SL_2$. 
\end{remark}

To generalize the integral $I_2(s, \varphi, \chi)$ for $\GL_2\times\GL_1$ to $\GL_n\times \GL_1$, one needs to apply a projection operator from cusp forms on $\GL_n$ to cusp forms on the mirabolic subgroup of $\GL_2$. This projection is realized by taking certain Fourier coefficient. When $n=2$, this projection is the identity operator. Let $N_{\GL_n}$ be the maximal unipotent subgroup of $\GL_n$, and let $N_{\GL_n}^0=\{(u_{i,j})\in N_{\GL_n}: u_{1,2}=0\}$. We extend the character $\psi$ on $F\backslash \A$ to $N_{\GL_n}(F)\backslash N_{\GL_n}(\A)$ by \begin{equation*}
    \psi(u)=\psi(u_{1,2}+\cdots + u_{n-1,n})
\end{equation*}for $u=(u_{i,j})\in N_{\GL_n}(\A)$. For a non-zero cusp form $\varphi\in V_\pi$  in the space of an irreducible automorphic cuspidal representation of $\GL_n(\A)$, and a unitary idele class character $\chi$, we consider the integral
\begin{equation}
\label{eq-GL(n)xGL(1)-global-integral}
I_n(s, \varphi, \chi)=\int\limits_{F^\times\backslash \A^\times}\int\limits_{N_{\GL_n}^0(F)\backslash N_{\GL_n}^0(\A)} \varphi\left( u \begin{pmatrix} 
t& & \\ & 1 & \\ & & I_{n-2}\end{pmatrix} \right) \psi^{-1}(u) \chi(t)|t|^{s-\frac{n-1}{2}} du d^\times t.
\end{equation}
The integral $I_n(s, \varphi, \chi)$ unfolds to (when $\Re(s)\gg 0$)
\begin{equation}
\label{eq-GL(n)xGL(1)-global-integral-unfolding}
\begin{split}
    \int\limits_{\A^\times} W_{\varphi, \psi} \left(\begin{pmatrix} 
t & \\ & I_{n-1}\end{pmatrix} \right) \chi(t)|t|^{s-\frac{n-1}{2}}d^\times t
\end{split}
\end{equation}
where
\begin{equation*}
    W_{\varphi, \psi}(g)=\int\limits_{N_{\GL_n}(F)\backslash N_{\GL_n}(\A)}\varphi(ug)\psi^{-1}(u)du,
\end{equation*}
and it represents $L^S(s, \pi\times\chi)$.
More generally, the theory of $L$-functions for $\GL_n\times\GL_m$ was developed in a series of papers \cite{MR618323, MR623137}.
For more details we refer the readers to the wonderful lecture notes \cite{MR2071506}.

The purpose of this paper is to present an integral representation of Hecke type, similar to \eqref{eq-GL(n)xGL(1)-global-integral}, for the twisted standard $L$-function for $\Sp_{2n}\times \GL_1$ of degree $2n+1$. 
Let $\pi$ be an irreducible automorphic cuspidal representation of $\Sp_{2n}(\A)$. Throughout this paper, we always assume $\pi$ is generic. To recall this notion, let $N$ be the standard maximal unipotent subgroup of $\Sp_{2n}$. We define a character on $N(F)\backslash N(\A)$, denoted as $\psi_N$, by
\begin{equation*}
    \psi_N(u)=\psi(\sum_{i=1}^n u_{i, i+1}), \, u=(u_{i,j})\in N(\A).
\end{equation*}
Then $\pi$ is generic if the space of functions
\begin{equation*}
    W_{\varphi, \psi_N}(g)=\int\limits_{N(F)\backslash N(\A)} \varphi(ug)\psi_N^{-1}(u)du, \, g\in \Sp_{2n}(\A),
\end{equation*}
as $\varphi\in V_\pi$ varies, is not zero. 
Let $U$ be the following subgroup of $N$:
\begin{equation*}
    U=\left\{ u=(u_{i,j})\in N: u_{1,2}=u_{2n-1,2n}=0 \right\}.
\end{equation*}
Let $\alpha:\GL_1\to \Sp_{2n}$ be the embedding
\begin{equation*}
    \alpha(t)=\begin{pmatrix} t & &\\
    & I_{2n-2} & \\ & & t^{-1}\end{pmatrix}.
\end{equation*}
For a non-zero cusp form $\varphi\in V_\pi$ and a unitary idele class character $\chi$, we consider the following integral
\begin{equation}
\label{eq-global-integral}
\mathcal{Z}(s, \varphi, \chi)=\int\limits_{F^\times\backslash \A^\times} \int\limits_{U(F)\backslash U(\A)}\varphi( u\alpha(t)) \psi_N^{-1}(u)  \chi(t) |t|^{2s-n-\frac{1}{2}} dud^\times t.
\end{equation}

In this note we prove the following theorem. 
\begin{theorem}
\label{theorem-main}
Let $\pi$ be an irreducible automorphic cuspidal generic representation of $\Sp_{2n}(\A)$ and let $\chi:F^\times\backslash \A^\times\to \C^\times$ be a unitary idele class character. Suppose $\varphi\in V_\pi$ is a non-zero cusp form which corresponds to $\otimes_v \xi_v$ under the tensor product decomposition $\pi\cong \otimes_v^\prime \pi_v$. Then
\begin{equation*}
    \mathcal{Z}(s, \varphi, \chi)=\frac{L^S(2s-\frac{1}{2}, \pi\times\chi)}{L^S(4s-1, \chi)}\cdot \mathcal{Z}_S(s, \varphi, \chi),
\end{equation*}
where $S$ is a finite set of places (including all the Archimedean ones such that for $v\not\in S$, $\pi_v$, $\chi_v$ and $\psi_v$ are unramified) and $\mathcal{Z}_S(s, \varphi, \chi)$ is a meromorphic function.
Moreover, for any $s_0\in \C$, the data can be chosen so that $\mathcal{Z}_S(s_0, \varphi, \chi)\not=0$.
\end{theorem}

\begin{remark}
When $n=1$, our integral becomes
\begin{equation*}
    \int\limits_{F^\times\backslash \A^\times} \varphi\left(\begin{smallmatrix} t&\\ & t^{-1}\end{smallmatrix}\right)\chi(t)|t|^{2s-\frac{3}{2}}d^\times t.
\end{equation*}
This is the integral \eqref{eq-SL(2)xGL(1)-global-integral} after a change of variable in $s$.
\end{remark}

\begin{remark}
When $n=2$, our integral becomes
\begin{equation}
\label{eq-Sp(4)xGL(1)-global-integral}
\int\limits_{F^\times\backslash \A^\times} \int\limits_{F^3\backslash \A^3} \varphi\left( \begin{pmatrix}
1 & & y & z\\
  & 1 &w &y\\
  & & 1 & \\
  &&&1
\end{pmatrix} \begin{pmatrix}
t & & &\\
&1 &&\\
&&1 & \\
&&&t^{-1}
\end{pmatrix}  \right)\psi^{-1}(w)\chi(t)|t|^{2s-\frac{5}{2}}dydzdwd^\times t,
\end{equation}
which represents the degree five twisted standard $L$-function for $\Sp_4\times\GL_1$.
It is interesting to compare \eqref{eq-Sp(4)xGL(1)-global-integral} with Novodvorsky's integral (\cite{MR546610}, see also \cite[Section 3.3]{MR993311})
\begin{equation}
\label{eq-GSp(4)xGL(1)-global-integral}
\int\limits_{F^\times\backslash \A^\times} \int\limits_{F^3\backslash \A^3} \phi\left( \begin{pmatrix}
1 &y &  & z\\
  & 1 & &\\
  & w& 1 &-y \\
  &&&1
\end{pmatrix} \begin{pmatrix}
t & & &\\
&t &&\\
&&1 & \\
&&&1
\end{pmatrix}  \right)\psi^{-1}(y)\chi(t)|t|^{s-\frac{1}{2}}dydzdwd^\times t
\end{equation}
where $\phi$ is a generic cusp form on $\GSp_4(\A)$,
which represents the degree four twisted Spin $L$-function for $\GSp_4\times\GL_1$.
\end{remark}

We emphasize that $L$-functions for symplectic groups for generic representations are very well-understood, for example from \cite{MR1454260} and \cite{MR1675971} in which Ginzburg, Rallis and Soudry constructed a family of Shimura type integrals which involve theta series. Our goal is to impress the reader the similarity between $\GL_n\times\GL_1$ and $\Sp_{2n}\times\GL_1$ in the context of Hecke type integrals.

Once formulated, Theorem~\ref{theorem-main} follows immediately from the standard procedure in the theory of integral representations of $L$-function. The global unfolding computation is discussed in Theorem~\ref{theorem-global-unfolding}. The local unramified computation is given in Proposition~\ref{proposition-local-unramified-computation}. The non-vanishing of the local integrals at ramified and archimedean places is discussed in Proposition~\ref{proposition-local-non-vanishing}. We remark that Proposition~\ref{proposition-local-unramified-computation} and Proposition~\ref{proposition-local-non-vanishing} can be deduced from the work in \cite{MR1675971}. Thus our main contribution in this paper is the global unfolding computation in Section~\ref{section-global-computation}.

\section{Global unfolding}
\label{section-global-computation}

In this section, we prove some basic global results about the integral $\mathcal{Z}(s,\varphi, \chi)$.

\begin{theorem}
\label{theorem-global-unfolding}
The integral $\mathcal{Z}(s,\varphi, \chi)$ converges absolutely when $\Re(s)\gg 0$ and can be meromorphically continued to all $s\in \C$. Moreover, when $\Re(s)\gg 0$, the integral $\mathcal{Z}(s,\varphi, \chi)$ unfolds to 
\begin{equation}
\label{eq-unfolding}
    \int\limits_{\A^\times} W_{\varphi,\psi_N}(\alpha(t)) \chi(t) |t|^{2s-n-\frac{1}{2}} d^\times t.
\end{equation}
\end{theorem}

The first statement is standard and we omit the details; see for example \cite[Theorem A(a)]{MR1041001}.
To prove the second statement, we first prove the following identity.
\begin{proposition}
\label{proposition-unfolding-identity}
For $g\in \Sp_{2n}(\A)$, we have
\begin{equation}
\label{eq-unfolding-identity}
\int\limits_{U(F)\backslash U(\A)} \varphi(ug)\psi_N^{-1}(u)du = \sum_{\xi\in F^\times} W_{\varphi,\psi_N}(\alpha(\xi) g).
\end{equation}
\end{proposition}

\begin{proof}
Let us fix $g\in \Sp_{2n}(\A)$ and consider the function $F_g$ defined by
\begin{equation*}
    F_g(x_0)=\int\limits_{U(F)\backslash U(\A)} \varphi\left(  \left(\begin{smallmatrix}
    1 & x_0 & 0 & \cdots & 0 & 0 & 0 \\
      & 1   & 0 & \cdots &  0 & 0 & 0 \\
      &    & 1 & \cdots &  0 & 0 & 0 \\
      &     &   & \ddots  &\vdots  & \vdots &\vdots\\
      &     &   &     &1  &0   &0 \\
      &     &   &     &  &1   &-x_0 \\
      &     &   &     &   &     &1
    \end{smallmatrix}\right)
    ug
    \right) \psi_N^{-1}(u)du.
\end{equation*}
This is a function on $F\backslash \A$ and hence it has a Fourier expansion
\begin{equation*}
     F_g(x_0)=\sum_{\xi\in F}\int\limits_{F\backslash \A} \int\limits_{U(F)\backslash U(\A)}  \varphi\left(  \left(\begin{smallmatrix}
    1 & x & 0 & \cdots & 0 & 0 & 0 \\
      & 1   & 0 & \cdots &  0 & 0 & 0 \\
      &    & 1 & \cdots &  0 & 0 & 0 \\
      &     &   & \ddots  &\vdots  & \vdots &\vdots\\
      &     &   &     &1  &0   &0 \\
      &     &   &     &  &1   &-x \\
      &     &   &     &   &     &1
    \end{smallmatrix}\right)
    ug
    \right) \psi_N^{-1}(u)du \psi^{-1}(x\xi) dx \cdot \psi(x_0 \xi).
\end{equation*}
Evaluating $F_g$ at $x_0=0$ and breaking the summation over $\xi\in F$ into two parts where one part corresponds to $\xi=0$ and the other part corresponds to the summation over $\xi\in F^\times$, we obtain
\begin{equation}
\label{eq-proposition-unfolding-eq1}
\begin{split}
    F_g(0) &=  \int\limits_{F\backslash \A} \int\limits_{U(F)\backslash U(\A)}  \varphi\left(  \left(\begin{smallmatrix}
    1 & x & 0 & \cdots & 0 & 0 & 0 \\
      & 1   & 0 & \cdots &  0 & 0 & 0 \\
      &    & 1 & \cdots &  0 & 0 & 0 \\
      &     &   & \ddots  &\vdots  & \vdots &\vdots\\
      &     &   &     &1  &0   &0 \\
      &     &   &     &  &1   &-x \\
      &     &   &     &   &     &1
    \end{smallmatrix}\right)
    ug
    \right) \psi_N^{-1}(u)du  dx \\
    & \, + \sum_{\xi\in F^\times}\int\limits_{F\backslash \A} \int\limits_{U(F)\backslash U(\A)}  \varphi\left(  \left(\begin{smallmatrix}
    1 & x & 0 & \cdots & 0 & 0 & 0 \\
      & 1   & 0 & \cdots &  0 & 0 & 0 \\
      &    & 1 & \cdots &  0 & 0 & 0 \\
      &     &   & \ddots  &\vdots  & \vdots &\vdots\\
      &     &   &     &1  &0   &0 \\
      &     &   &     &  &1   &-x \\
      &     &   &     &   &     &1
    \end{smallmatrix}\right)
    ug
    \right) \psi_N^{-1}(u)du \psi^{-1}(x\xi) dx.
\end{split}
\end{equation}
The first part on the right hand side of \eqref{eq-proposition-unfolding-eq1} contains
\begin{equation*}
\begin{split}
    \int\limits_{(F\backslash \A)^{2n-1}} \varphi\left(  \left(\begin{smallmatrix}
    1 & u_{1,2} & u_{1,3} & \cdots & u_{1,2n-2} & u_{1,2n-1} & u_{1,2n} \\
      & 1   & 0 & \cdots &  0 & 0 & u_{1,2n-1}^* \\
      &    & 1 & \cdots &  0 & 0 & u_{1,2n-2}^* \\
      &     &   & \ddots  &\vdots  & \vdots &\vdots\\
      &     &   &     &1  &0   &u_{1,3}^* \\
      &     &   &     &  &1   &u_{1,2}^* \\
      &     &   &     &   &     &1
    \end{smallmatrix}\right)
    \tilde{u} g
    \right) \prod_{i=2}^{2n} du_{1,i}
\end{split}
\end{equation*}
as an inner integration, which is equal to zero by the cuspidality of $\varphi$. Here, $\tilde{u}$ is the matrix of the form
\begin{equation*}
\left(\begin{smallmatrix}
    1 & 0 & 0 & \cdots & 0 & 0 & 0 \\
      & 1   & u_{2,3} & \cdots &u_{2,2n-2}   &u_{2,2n-1}   & 0 \\
      &    & 1 & \cdots & u_{3,2n-2}  &u_{3,2n-1}  & 0 \\
      &     &   & \ddots  &\vdots  & \vdots &\vdots\\
      &     &   &     &1  &u_{2,3}^*   &0 \\
      &     &   &     &  &1   &0 \\
      &     &   &     &   &     &1
    \end{smallmatrix}\right).
\end{equation*}
So
\begin{equation*}
    \int\limits_{F\backslash \A} \int\limits_{U(F)\backslash U(\A)}  \varphi\left(  \left(\begin{smallmatrix}
    1 & x & 0 & \cdots & 0 & 0 & 0 \\
      & 1   & 0 & \cdots &  0 & 0 & 0 \\
      &    & 1 & \cdots &  0 & 0 & 0 \\
      &     &   & \ddots  &\vdots  & \vdots &\vdots\\
      &     &   &     &1  &0   &0 \\
      &     &   &     &  &1   &-x \\
      &     &   &     &   &     &1
    \end{smallmatrix}\right)
    ug
    \right) \psi_N^{-1}(u)du  dx=0
\end{equation*}
and hence
\begin{equation*}
\begin{split}
    F_g(0) = \sum_{\xi\in F^\times}\int\limits_{F\backslash \A} \int\limits_{U(F)\backslash U(\A)}  \varphi\left(  \left(\begin{smallmatrix}
    1 & x & 0 & \cdots & 0 & 0 & 0 \\
      & 1   & 0 & \cdots &  0 & 0 & 0 \\
      &    & 1 & \cdots &  0 & 0 & 0 \\
      &     &   & \ddots  &\vdots  & \vdots &\vdots\\
      &     &   &     &1  &0   &0 \\
      &     &   &     &  &1   &-x \\
      &     &   &     &   &     &1
    \end{smallmatrix}\right)
    ug
    \right) \psi_N^{-1}(u)du \psi^{-1}(x\xi) dx.
\end{split}
\end{equation*}
Since $\varphi$ is invariant under left multiplication by $\Sp_{2n}(F)$, we have
\begin{equation*}
\begin{split}
    F_g(0) = \sum_{\xi\in F^\times}\int\limits_{F\backslash \A} \int\limits_{U(F)\backslash U(\A)}  \varphi\left(  
   \left(\begin{smallmatrix}
    \xi &  &  &  &  &  & \\
      & 1   &  & &   &  & \\
      &    & 1 &  &   &  &  \\
      &     &   & \ddots  &  &  &\\
      &     &   &     &1  &   & \\
      &     &   &     &  &1   & \\
      &     &   &     &   &     &\xi^{-1}
   \end{smallmatrix}\right)
    \left(\begin{smallmatrix}
    1 & x & 0 & \cdots & 0 & 0 & 0 \\
      & 1   & 0 & \cdots &  0 & 0 & 0 \\
      &    & 1 & \cdots &  0 & 0 & 0 \\
      &     &   & \ddots  &\vdots  & \vdots &\vdots\\
      &     &   &     &1  &0   &0 \\
      &     &   &     &  &1   &-x \\
      &     &   &     &   &     &1
    \end{smallmatrix}\right)
    ug
    \right) \psi_N^{-1}(u)du \psi^{-1}(x\xi) dx.
\end{split}
\end{equation*}
We then use the following matrix identity
\begin{equation*}
   \alpha(\xi)
    \left(\begin{smallmatrix}
    1 & x & 0 & \cdots & 0 & 0 & 0 \\
      & 1   & 0 & \cdots &  0 & 0 & 0 \\
      &    & 1 & \cdots &  0 & 0 & 0 \\
      &     &   & \ddots  &\vdots  & \vdots &\vdots\\
      &     &   &     &1  &0   &0 \\
      &     &   &     &  &1   &-x \\
      &     &   &     &   &     &1
    \end{smallmatrix}\right)
=
        \left(\begin{smallmatrix}
    1 & x\xi & 0 & \cdots & 0 & 0 & 0 \\
      & 1   & 0 & \cdots &  0 & 0 & 0 \\
      &    & 1 & \cdots &  0 & 0 & 0 \\
      &     &   & \ddots  &\vdots  & \vdots &\vdots\\
      &     &   &     &1  &0   &0 \\
      &     &   &     &  &1   &-x\xi \\
      &     &   &     &   &     &1
    \end{smallmatrix}\right)
     \alpha(\xi)
\end{equation*}
and perform a change of variable $x\mapsto x\xi^{-1}$ to get
\begin{equation*}
\begin{split}
    F_g(0) = \sum_{\xi\in F^\times}\int\limits_{F\backslash \A} \int\limits_{U(F)\backslash U(\A)}  \varphi\left(  
    \left(\begin{smallmatrix}
    1 & x & 0 & \cdots & 0 & 0 & 0 \\
      & 1   & 0 & \cdots &  0 & 0 & 0 \\
      &    & 1 & \cdots &  0 & 0 & 0 \\
      &     &   & \ddots  &\vdots  & \vdots &\vdots\\
      &     &   &     &1  &0   &0 \\
      &     &   &     &  &1   &-x \\
      &     &   &     &   &     &1
    \end{smallmatrix}\right)
  \alpha(\xi)
    ug
    \right) \psi_N^{-1}(u)du \psi^{-1}(x) dx.
\end{split}
\end{equation*}
Next, we conjugate $\alpha(\xi)$ to the right of the matrix $u$. Notice that the conjugation by $\alpha(\xi)$ does not affect the character $\psi_N^{-1}$, i.e., $\psi_N^{-1}(\alpha(\xi)u\alpha(\xi)^{-1})=\psi_N^{-1}(u)$. After a change of variables in $u$, we obtain
\begin{equation*}
\begin{split}
    F_g(0)  &= \sum_{\xi\in F^\times}\int\limits_{F\backslash \A} \int\limits_{U(F)\backslash U(\A)}  \varphi\left(  
    \left(\begin{smallmatrix}
    1 & x & 0 & \cdots & 0 & 0 & 0 \\
      & 1   & 0 & \cdots &  0 & 0 & 0 \\
      &    & 1 & \cdots &  0 & 0 & 0 \\
      &     &   & \ddots  &\vdots  & \vdots &\vdots\\
      &     &   &     &1  &0   &0 \\
      &     &   &     &  &1   &-x \\
      &     &   &     &   &     &1
    \end{smallmatrix}\right)
    u   \alpha(\xi)g
    \right) \psi_N^{-1}(u)du \psi^{-1}(x) dx.
\end{split}
\end{equation*}
Hence
\begin{equation*}
\begin{split}
    F_g(0)  &= \sum_{\xi\in F^\times} \int\limits_{N(F)\backslash N(\A)} \varphi(n \alpha(\xi)g)\psi_N^{-1}(n)dn = \sum_{\xi\in F^\times} W_{\varphi, \psi_N}(\alpha(\xi)g).
\end{split}
\end{equation*}
Finally, notice that $F_g(0)$ is equal to the left hand side of \eqref{proposition-unfolding-identity}. This completes the proof.
\end{proof}

Now we are ready to finish the proof of Theorem~\ref{theorem-global-unfolding}.

\begin{proof}[Proof of Theorem~\ref{theorem-global-unfolding}]
By Proposition~\ref{proposition-unfolding-identity}, we have
\begin{equation*}
\begin{split}
\mathcal{Z}(s, \varphi, \chi) &=\int\limits_{F^\times\backslash \A^\times} \int\limits_{U(F)\backslash U(\A)}\varphi( u\alpha(t)) \psi_N^{-1}(u)  \chi(t) |t|^{2s-n-\frac{1}{2}} dud^\times t\\
&=\int\limits_{F^\times\backslash \A^\times} \sum_{\xi\in F^\times} W_{\varphi,\psi_N}(\alpha(\xi) \alpha(t)) \chi(t)|t|^{2s-n-\frac{1}{2}}d^\times t.
\end{split}
\end{equation*}
For any $\xi\in F^\times$, we have $\chi(t)|t|^{2s-n-\frac{1}{2}}=\chi(\xi t)|\xi t|^{2s-n-\frac{1}{2}}$. Then
\begin{equation*}
\begin{split}
\mathcal{Z}(s, \varphi, \chi) =\int\limits_{F^\times\backslash \A^\times} \sum_{\xi\in F^\times} W_{\varphi,\psi_N}(\alpha(\xi t)) \chi(\xi t)|\xi t|^{2s-n-\frac{1}{2}}d^\times t.
\end{split}
\end{equation*}
Finally we collapse the summation with the integration to get
\begin{equation*}
\begin{split}
\mathcal{Z}(s, \varphi, \chi) =  \int\limits_{\A^\times} W_{\varphi,\psi_N}(\alpha(t)) \chi(t) |t|^{2s-n-\frac{1}{2}} d^\times t.   
\end{split}
\end{equation*}
\end{proof}

\section{Local unramified computation}
\label{section-local-computation}

In this section, we compute the local unramified integral corresponding to the global integral \eqref{eq-unfolding} in Theorem~\ref{theorem-global-unfolding}. Let $F$ be a non-archimedean local field with ring $\mathcal{O}$ of integers, with a fixed uniformizer $\varpi$. Let $q$ be the cardinality of the residue field. The absolute value $|\cdot|$ on $F$ is normalized so that $|\varpi|=q^{-1}$. Let $\psi$ be a non-trivial additive unramified character of $F$. We assume $\pi$ is an irreducible admissible unramified generic representation of $\Sp_{2n}(F)$, with Satake parameters $a_1, a_2, \cdots, a_n\in \C^\times$.
By assumption, there is a unique unramified Whittaker function $W_\pi^0\in \mathcal{W}(\pi, \psi_N)$, where $\mathcal{W}(\pi, \psi_N)$ is the space of Whittaker functions associated to $\pi$ with respect to the character $\psi_N$, so that $W_\pi^0(k)=W_\pi^0(I_{2n})=1$ for any $k\in \Sp_{2n}(\mathcal{O})$.
We assume $\chi$ is an unramified character of $F^\times$. The local $L$-function associated to $\chi$ is
\begin{equation*}
    L(s,\chi)=(1-\chi(\varpi)q^{-s})^{-1}.
\end{equation*}
Note that the $L$-group of $\Sp_{2n}$ is $\SO_{2n+1}(\C)$. Let 
\begin{equation*}
    A=\mathrm{diag}(a_1, \cdots, a_n, 1, a_n^{-1}, \cdots, a_1^{-1})
\end{equation*}
be the semisimple conjguacy class in $\SO_{2n+1}(F)$. Then the local twisted standard $L$-function for $\pi\times\chi$ is
\begin{equation*}
    L(s,\pi\times \chi)=\mathrm{det}(1-A \chi(\varpi)q^{-s})^{-1}.
\end{equation*}

We normalize the Haar measure on the multiplicative group $F^\times$ so that $\int\limits_{\mathcal{O}^\times}d^\times t=1$.
Let
\begin{equation}
\label{eq-local-unramified-integral}
    \mathcal{Z}(s,W_\pi^0, \chi)=\int\limits_{F^\times}W_\pi^0(\alpha(t))\chi(t)|t|^{2s-n-\frac{1}{2}}d^\times t.
\end{equation}
This is the local unramified integral corresponding to \eqref{eq-unfolding}. We have the following result on the local unramified integral.

\begin{proposition}
\label{proposition-local-unramified-computation}
For $\Re(s)\gg 0$, we have
\begin{equation*}
\mathcal{Z}(s,W_\pi^0, \chi)=\frac{L(2s-\frac{1}{2}, \pi\times\chi)}{L(4s-1, \chi)}.
\end{equation*}
\end{proposition}

\begin{proof}
This follows from the local unramified computation in \cite{MR1675971}. Indeed, the integral $\mathcal{Z}(s,W_\pi^0, \chi)$ in \eqref{eq-local-unramified-integral} is exactly the last integral in \cite[pp. 196]{MR1675971} in the special case when $k=1$, and the integral was computed by using the Casselman-Shalika formula \cite{MR581582}.
\end{proof}

\begin{remark}
We remark that the definition of the local unramified integral in \cite{MR1675971} involves an unramified section of certain induced representation. After applying the Iwasawa decomposition for $\SL_2$ and some computation,
the local unramified integral (for $\Sp_{2n}\times\GL_1$) in \cite{MR1675971} coincide with our integral $\mathcal{Z}(s,W_\pi^0, \chi)$ in \eqref{eq-local-unramified-integral}.
\end{remark}

\begin{remark}
We would like to point out that when $\chi=1$ is the trivial character, the integral $\mathcal{Z}(s,W_\pi^0, \chi)$ was also computed in \cite[pp. 94--95]{MR1454260}.
\end{remark}

\section{Local non-vanishing results}

In this section, we prove that, given any $s_0\in \C$, there is a choice of data so that the local integrals are non-zero at $s_0$.

Let $F$ be a local field, which can be archimedean or non-archimedean. Let $\pi$ be an irreducible generic representation of $\Sp_{2n}(F)$, and let $\chi$ be a unitary character of $F^\times$. The space of Whittaker functions for $\pi$ is denoted by $\mathcal{W}(\pi, \psi_N)$. Our local integral is
\begin{equation}
    \mathcal{Z}(s, W_\pi, \chi)= \int\limits_{F^\times}W_\pi(\alpha(t))\chi(t)|t|^{2s-n-\frac{1}{2}}d^\times t
\end{equation}
where $W_\pi\in \mathcal{W}(\pi, \psi_N)$. 
We have the following results.

\begin{proposition}
\label{proposition-local-non-vanishing}
\begin{enumerate}[(i)]
    \item When $\Re(s)\gg 0$, the integral $\mathcal{Z}(s, W_\pi, \chi)$ converges absolutely.
    \item The integral $\mathcal{Z}(s, W_\pi, \chi)$ admits a meromorphic continuation to the entire complex plane. 
    
    \item Given $s_0\in \C$, there is a choice of data so that $\mathcal{Z}(s, W_\pi, \chi)$ is non-zero at $s_0$.
\end{enumerate}
\end{proposition}

\begin{proof}
This follows from \cite[Lemma 3.4, Lemma 3.5, Proposition 3.6]{MR1675971}.
\end{proof}

\begin{remark}
We remark that Proposition~\ref{proposition-local-non-vanishing}(i) was proved by the asymptotic expansion for Whittaker functions $W_\pi\in \mathcal{W}(\pi, \psi_N)$; see \cite[Lemma 3.3]{MR1675971}. Proposition~\ref{proposition-local-non-vanishing}(iii) was  proved by contradiction: if we assume that $\mathcal{Z}(s, W_\pi, \chi)$ is zero for all choices of data, then one can deduce that $W_\pi(I_{2n})\chi(1)$ is zero for all Whittaker functions $W_\pi\in \mathcal{W}(\pi, \psi_N)$ and unitary characters $\chi$, which would lead to a contradiction.
\end{remark}

\section{Proof of Theorem~\ref{theorem-main}}

We now prove Theorem~\ref{theorem-main}.

\begin{proof}[Proof of Theorem~\ref{theorem-main}]
Choose $S$ as in Theorem~\ref{theorem-main} such that for all $v\not\in S$, $\xi_v=\xi_v^0$ is the unramified vector in $V_{\pi_v}$.
By Theorem~\ref{theorem-global-unfolding}, we have
\begin{equation*}
\begin{split}
     \mathcal{Z}(s, \varphi, \chi)
            =  \int\limits_{\A^\times} W_{\varphi,\psi_N}(\alpha(t)) \chi(t) |t|^{2s-n-\frac{1}{2}} d^\times t    = \prod_{v}  \int\limits_{F_v^\times} W_{\xi_v,\psi_N}(\alpha(t_v)) \chi_v(t_v) |t_v|^{2s-n-\frac{1}{2}} d^\times t_v .
\end{split}
\end{equation*}
By Proposition~\ref{proposition-local-unramified-computation}, we know that for $v\not\in S$, we have
\begin{equation*}
    \int\limits_{F_v^\times} W_{\xi_v^0,\psi_N}(\alpha(t_v)) \chi_v(t_v) |t_v|^{2s-n-\frac{1}{2}} d^\times t_v = \frac{L(2s-\frac{1}{2}, \pi_v\times\chi_v)}{L(4s-1, \chi_v)}.
\end{equation*}
Denote
\begin{equation*}
    \mathcal{Z}_S(s, \varphi, \chi)= \prod_{v\in S}  \int\limits_{F_v^\times} W_{\xi_v,\psi_N}(\alpha(t_v)) \chi_v(t_v) |t_v|^{2s-n-\frac{1}{2}} d^\times t_v .
\end{equation*}
Hence we have 
\begin{equation*}
\begin{split}
     \mathcal{Z}(s, \varphi, \chi)
             =& \prod_{v\not\in S}  \int\limits_{F_v^\times} W_{\xi_v,\psi_N}(\alpha(t_v)) \chi_v(t_v) |t_v|^{2s-n-\frac{1}{2}} d^\times t_v  \cdot   \mathcal{Z}_S(s, \varphi, \chi) \\
            =& \frac{L^S(2s-\frac{1}{2}, \pi\times\chi)}{L^S(4s-1, \chi)}\cdot \mathcal{Z}_S(s, \varphi, \chi).
\end{split}
\end{equation*}
The claimed properties of $\mathcal{Z}_S(s, \varphi, \chi)$ follow from Proposition~\ref{proposition-local-non-vanishing}. This completes the proof of Theorem~\ref{theorem-main}.
\end{proof}

\section*{Acknowledgements} 

This work was completed while I was a PhD student at The Ohio State University. I thank my advisor Jim Cogdell for his support and encouragement.

\def\cprime{$'$} \def\Dbar{\leavevmode\lower.6ex\hbox to 0pt{\hskip-.23ex
  \accent"16\hss}D} \def\cftil#1{\ifmmode\setbox7\hbox{$\accent"5E#1$}\else
  \setbox7\hbox{\accent"5E#1}\penalty 10000\relax\fi\raise 1\ht7
  \hbox{\lower1.15ex\hbox to 1\wd7{\hss\accent"7E\hss}}\penalty 10000
  \hskip-1\wd7\penalty 10000\box7}
  \def\polhk#1{\setbox0=\hbox{#1}{\ooalign{\hidewidth
  \lower1.5ex\hbox{`}\hidewidth\crcr\unhbox0}}} \def\dbar{\leavevmode\hbox to
  0pt{\hskip.2ex \accent"16\hss}d}
  \def\cfac#1{\ifmmode\setbox7\hbox{$\accent"5E#1$}\else
  \setbox7\hbox{\accent"5E#1}\penalty 10000\relax\fi\raise 1\ht7
  \hbox{\lower1.15ex\hbox to 1\wd7{\hss\accent"13\hss}}\penalty 10000
  \hskip-1\wd7\penalty 10000\box7}
  \def\ocirc#1{\ifmmode\setbox0=\hbox{$#1$}\dimen0=\ht0 \advance\dimen0
  by1pt\rlap{\hbox to\wd0{\hss\raise\dimen0
  \hbox{\hskip.2em$\scriptscriptstyle\circ$}\hss}}#1\else {\accent"17 #1}\fi}
  \def\bud{$''$} \def\cfudot#1{\ifmmode\setbox7\hbox{$\accent"5E#1$}\else
  \setbox7\hbox{\accent"5E#1}\penalty 10000\relax\fi\raise 1\ht7
  \hbox{\raise.1ex\hbox to 1\wd7{\hss.\hss}}\penalty 10000 \hskip-1\wd7\penalty
  10000\box7} \def\lfhook#1{\setbox0=\hbox{#1}{\ooalign{\hidewidth
  \lower1.5ex\hbox{'}\hidewidth\crcr\unhbox0}}}
\providecommand{\bysame}{\leavevmode\hbox to3em{\hrulefill}\thinspace}
\providecommand{\MR}{\relax\ifhmode\unskip\space\fi MR }
\providecommand{\MRhref}[2]{%
  \href{http://www.ams.org/mathscinet-getitem?mr=#1}{#2}
}
\providecommand{\href}[2]{#2}


\begin{thebibliography}{99}


\bibitem[Bum89]{MR993311}
Bump, Daniel,
\emph{The {R}ankin-{S}elberg method: a survey},
   Number theory, trace formulas and discrete groups ({O}slo,
              1987), 49--109, Academic Press, Boston, MA, 1989. \MR{993311}


\bibitem[CS80]{MR581582}
Casselman, W. and Shalika, J.,
\emph{The unramified principal series of {$p$}-adic groups. {II}. {T}he {W}hittaker function},
   Compositio Math., \textbf{41},
   (1980), no.~2, 207--231. \MR{581582}

\bibitem[Cog04]{MR2071506}
Cogdell, James W., \emph{Lectures on {$L$}-functions, converse theorems, and
              functoriality for {${\rm GL}_n$}}, Lectures on automorphic {$L$}-functions, Fields Inst. Monogr., Vol. 20, 1--96, Amer. Math. Soc., Providence, RI, 2004. \MR{2071506}  

\bibitem[Hec36]{MR1513069}
Hecke, E.,
\emph{\"{U}ber die Bestimmung Dirichletscher Reihen durch ihre
   Funktionalgleichung},
   Math. Ann., \textbf{112},
   (1936), no.~1, 664--699. \MR{1513069}


\bibitem[Gin90]{MR1041001}
Ginzburg, David,
\emph{{$L$}-functions for {${\rm SO}_n\times {\rm GL}_k$}}, J. Reine Angew. Math., \textbf{405} (1990), 156--180. \MR{1041001}


\bibitem[GRS97]{MR1454260}
David Ginzburg, Stephen Rallis, and David Soudry, \emph{Periods, poles of $L$-functions and symplectic-orthogonal theta
   lifts}, J. Reine Angew. Math., \textbf{487} (1997), 85--114. \MR{1454260}

\bibitem[GRS98]{MR1675971}
David Ginzburg, Stephen Rallis, and David Soudry, \emph{$L$-functions for symplectic groups}, Bull. Soc. Math. France, \textbf{126} (1998), no.~2, 181--244. \MR{1675971}
  
\bibitem[JL70]{MR0401654}
Jacquet, H. and Langlands, R. P., \emph{Automorphic forms on {${\rm GL}(2)$}}, Lecture Notes in Mathematics, Vol. 114, Springer-Verlag, Berlin-New York, 1970. \MR{0401654}  
  


\bibitem[JS81A]{MR618323}
Jacquet, H., Shalika, J. A., \emph{On Euler products and the classification of automorphic
   representations. I},  Amer. J. Math., \textbf{103}, no.~3, 499--558, 1981. \MR{618323}


\bibitem[JS81B]{MR623137}
Jacquet, H., Shalika, J. A., \emph{On Euler products and the classification of automorphic
   representations. II},  Amer. J. Math., \textbf{103}, no.~4, 777--815, 1981. \MR{623137}


\bibitem[Nov79]{MR546610}
Novodvorsky, Mark E., \emph{Automorphic {$L$}-functions for symplectic group {${\rm
              GSp}(4)$}}, Automorphic forms, representations and {$L$}-functions
              ({P}roc. {S}ympos. {P}ure {M}ath., {O}regon {S}tate {U}niv.,
              {C}orvallis, {O}re., 1977), {P}art 2, Proc. Sympos. Pure Math., XXXIII, 87--95, Amer. Math. Soc., Providence, R.I., 1979. \MR{546610} 


\bibitem[Sha74]{MR348047}
Shalika, J. A., \emph{The multiplicity one theorem for ${\rm GL}_{n}$},         Ann. of Math. (2), volume 100, 171--193, 1974. \MR{348047} 







\end{thebibliography}
\end{document}